\documentclass[11pt]{amsart}
\usepackage{amsmath}
\usepackage{amsfonts}
\usepackage{amssymb,enumerate}
\usepackage{amsthm}
\usepackage{hyperref}
\usepackage{centernot}
\usepackage{stmaryrd}

\DeclareMathOperator\hind{Hind}
\DeclareMathOperator\add{add}
\DeclareMathOperator\pow{pow_2}

\theoremstyle{plain}
\newtheorem{theorem}{Theorem}[section]

\begin{document}

\title{Finite Sums of Arithmetic Progressions}
\author{Shahram Mohsenipour}

\address{School of Mathematics, Institute for Research in Fundamental Sciences (IPM)
        P. O. Box 19395-5746, Tehran, Iran}\email{sh.mohsenipour@gmail.com}


\begin{abstract} We give a purely combinatorial proof for a two-fold generalization of van der Waerden-Brauer's theorem and Hindman's theorem. We also give tower bounds for a finite version of it.
\end{abstract}
\maketitle
\bibliographystyle{amsplain}
\section{Introduction}
Let $l\geq 3$ be a positive integer and let $P=\{a_1,\dots,a_l\}$ be an $l$-term arithmetic progression with $a_1<\dots<a_l$, we denote the $s$th term of $P$ by $P[s]=a_s$. Now let $P$ and $Q$ be two $l$-term arithmetic progressions, we define their pointwise sum (or briefly their sum) $P\oplus Q$ as the $l$-term arithmetic progression with $P\!\oplus\! Q \,[s]=P[s]+Q[s]$ for $1\leq s\leq l$. Hence for the $l$-term arithmetic progressions $P_1,P_2,\dots,P_m$, their finite sum $P_1\oplus P_2\oplus\dots \oplus P_m$ has unambiguous meaning. The following pleasant two-fold generalization of van der Waerden-Brauer's theorem and Hindman's theorem, can be deduced form either Furstenberg's theorem (\cite{furs} Proposition 8.2.1) or Deuber-Hindman's theorem \cite{dubhind}. \\

{\em For any positive integers $c$ and $l\geq3$, if $\mathbb{N}$ is $c$-colored, then there exist a color $\gamma$ and infinitely many $l$-term arithmetic progressions $Q_i$, $i\in\mathbb{N}$ such that all of their finite sums (with no repetition) are monochromatic with the color $\gamma$ and all the common differences of the above finite sums have also the color $\gamma$ too}.\\

It would be pleasant too if we have a purely combinatorial proof of such a statement avoiding topological dynamics as well as the theory of ultrafilters. In Theorems \ref{maintheorem} and \ref{secondtheorem} of this paper we give such a proof. It is interesting to see whether the method of the proof can be generalized to give a combinatorial proof of Deuber-Hindman's theorem \cite{dubhind}. We are also interested in a finite version of the above theorem. It is well known that through a compactness argument we can have a finite version. For instance we have the following theorem which is a two-fold generalization of van der Waerden's theorem and a finite version of Hindman's theorem.\\

{\em For positive integers $c,n$ and $l\geq3$ there is a positive integer $m$ such that whenever $\{1,2,\dots,m\}$ is $c$-colored, then there exist $l$-term arithmetic progressions $P_1,P_2,\dots,P_n\subset\{1,2,\dots,m\}$ such that $\sum_{i=1}^{n}P_i[l]$ is not bigger than $m$ and all finite sums of $P_i$ (with no repetition) are monochromatic with the same color.}\\

If we denote the least such $m$ by $f(l,n,c)$ then the proof given through the compactness argument does not give us upper bounds for $f(l,n,c)$. But it is not hard to see that the proof given for Theorem \ref{maintheorem} can be made finitary (which may be regarded as an advantage of the proof over its counterparts using dynamical system or ultrafilters) to give us a primitive recursive upper bound for $f(l,n,c)$. To do so we use the finitary Hindman numbers $\textrm{Hind}(n,c)$ which is a tower function \cite{dodos}. However due to its iterated use of the function $\textrm{Hind}(n,c)$, it gives us an upper bound belonging to the class of WOW functions \cite{grs}. In Theorem \ref{thirdtheorem}, we do a better job by giving a different proof which uses the function $\textrm{Hind}(n,c)$ just one time and thus obtaining tower bounds for $f(l,n,c)$. Also note that according to the Gowers elementary bounds for the van der Waerden theorem, we don't worry about the van der Waerden part of the proof.

\section{Preliminaries}
Let's fix some notations. For $n$ a positive integer put $[n]=\{1,2,\dots,n\}$. Let $S$ be an infinite set, we denote the collection of finite nonempty subsets of $S$ by $\mathcal{P}_{f}(S)$. For a finite set $A$, $\mathcal{P}^+(A)$ denotes the collection of nonempty subsets of $A$. Also $FS(S)$ will denote the set of all finite sums of elements of $S$ with no repetition. Let $A,B\in\mathcal{P}_{f}(\mathbb{N})$, by $A<B$ we mean that $\max A < \min B$. We also denote the common difference of the arithmetic progression $P$ by $\add{P}$. We use the following notation for finite sums of arithmetic progressions
\[
\displaystyle\bigoplus_{i\in B}P_i=P_1\oplus P_2\oplus\dots \oplus P_m
\]
where $B=\{1,2,\dots,m\}$. Obviously we have
\[
\displaystyle\bigoplus_{i\in B}\!P_i\,[s]=\displaystyle\sum_{i\in B} P_i[s].
\]
We define a partial ordering between $l$-term arithmetic progressions by putting $P\prec Q$ whenever $P[s]<Q[s]$ for all $1\leq s\leq l$.
Let's state van der Waerden's theorem and van der Waerden-Brauer's theorem \cite{grs}.

\begin{theorem}[van der Waerden]
For positive integers $c$ and $l\geq 3$ there is a positive integer $n$ such that whenever $[n]$ is $c$-colored, then there is a monochromatic $l$-term arithmetic progression $P\subseteq[n]$. We denote the least such $n$ by $W(l,c)$.
\end{theorem}

\begin{theorem}[van der Waerden-Brauer]
For positive integers $c$ and $l\geq 3$ there is a positive integer $n$ such that whenever $\mathbf{c}$ is a $c$-coloring of $[n]$, then there are $d, a, a+d,\dots, a+(l-1)d$ in $\{1,2,\dots,n\}$ such that
\[
\mathbf{c}(d)=\mathbf{c}(a)=\mathbf{c}(a+d)=\dots=\mathbf{c}(a+(l-1)d).
\]
We denote the least such $n$ by $WB(l,c)$.
\end{theorem}

We will use the following strong version of Hindman's theorem \cite{taylor2}.

\begin{theorem}
Let $a_1<a_2<\dots<a_m<\dots$ be an infinite strictly increasing sequence of positive integers. Let $c$ be a positive integer and $FS(\{a_1,a_2,\dots\})$ be $c$-colored. Then there are $B_1<B_2<B_3<\dots$ in $\mathcal{P}_{f}(\mathbb{N})$ such that whenever
\[
b_1=\displaystyle\sum_{i\in B_1}a_i\,\,,\,\,\, b_2=\displaystyle\sum_{i\in B_2}a_i\,\,,\,\,\,\dots\,\,\,, b_m=\displaystyle\sum_{i\in B_m}a_i\,\,,\dots
\]
then $FS(\{b_1,b_2,\dots\})$ is monochromatic.
\end{theorem}

We say that the two positive integers $a,b$ are {\em power-disjoint}, if the powers occurring in the expansions of $a,b$ in base $2$ are disjoint sets, more precisely if we write $a=2^{k_1}+\dots+2^{k_m}$ and $b=2^{l_1}+\dots+2^{l_n}$, then the two sets $\{k_1,\dots,k_m\}$ and $\{l_1,\dots,l_n\}$ are disjoint. We denote the set $\{k_1,\dots,k_m\}$ by $\pow(a)$. We will use the following finitary version of Hindman's theorem \cite{dodos} which strengthens the Disjoint Unions Theorem. First we introduce a notation. If $T$ is a collection of pairwise disjoint sets, then $NU(T)$ will denote the set of non-empty unions of elements of $T$.

\begin{theorem}\label{dut}
For positive integers $n,c$ there is a positive integer $m$ such that for any $m$-element set $A=\{a_1,\dots,a_m\}$ of pairwise power-disjoint positive integers, whenever $\mathbf{c}$ is a $c$-coloring of $FS(A)$, then there exist $\gamma\in[c]$ and $B_1,\dots,B_n$ in $\mathcal{P}^+([m])$ such that
$B_1<\cdots<B_n$ and for all $C\in NU\{B_1,\dots,B_n\}$ we have
\[
\mathbf{c}\big{(}\displaystyle\sum_{i\in C}a_i\big{)}=\gamma.
\]
Moreover if $Hind(n,c)$ denotes the least such $m$, then $Hind(n,c)$ is a tower function.
\end{theorem}

\section{Purely Combinatorial Proofs}
In the following theorem we give a purely combinatorial proof of the two-fold generalization van der Waerden's theorem and Hindman's theorem mentioned in the introduction.
\begin{theorem}\label{maintheorem}
Let $c$ and $l\geq 3$ be positive integers. Let $\mathbf{c}$ be a $c$-coloring of $\mathbb{N}$, then there are $l$-term arithmetic progressions $Q_1, Q_2,Q_3,\dots$ such that
\begin{itemize}
\item[(i)] $Q_1\prec Q_2\prec Q_3\prec\cdots$,
\item[(ii)] there is $\gamma\in[c]$ such that for all $C\in\mathcal{P}_{f}(\mathbb{N})$ and all $s\in\{1,\dots,l\}$ we have
\[
\mathbf{c}\big{(}\displaystyle\bigoplus_{i\in C}Q_i\,[s]\big{)}=\gamma.
\]
\end{itemize}
\end{theorem}
\begin{proof}
Let $n=W(l,c)$ and let $a_1<a_2<\cdots<a_m<\cdots$ be a strictly increasing sequence of positive integers with $a_{m+1}>a_1+\dots+a_m+mn$.  For $i\in\mathbb{N}$ we put
\[
P^{0}_{i}=\{a_i, a_i+1,\dots,a_i+(n-1)\}.
\]
Obviously $P^{0}_{i}$ is an $n$-term arithmetic progression and we have
\[
P^{0}_{1}\prec P^{0}_{2}\prec P^{0}_{3}\prec\cdots
\]
In fact it is easily seen that for any $C_1<C_2$ in $\mathcal{P}_{f}(\mathbb{N})$ we have
\begin{equation}\label{new}
\displaystyle\bigoplus_{i\in C_1}P_i^0\prec\displaystyle\bigoplus_{i\in C_2}P_i^0.
\end{equation}

Now for $1\leq k\leq n$ we inductively define the $n$-term arithmetic progressions $P_1^k,P_2^k,P_3^k,\ldots$ so that there are $\alpha_1^k,\alpha_2^k,\dots,\alpha_k^k\in[c]$ such that the following two conditions are satisfied
\begin{itemize}
\item[(a)] for all $C\in\mathcal{P}_{f}(\mathbb{N})$ and all $s\in\{1,\dots,k\}$ we have
\[
\mathbf{c}\big{(}\displaystyle\bigoplus_{i\in C}P_i^k\,[s]\big{)}=\alpha_s^k,
\]
\item[(b)] for all $C_1<C_2$ in $\mathcal{P}_{f}(\mathbb{N})$ we have
\[
\displaystyle\bigoplus_{i\in C_1}P_i^k\prec\displaystyle\bigoplus_{i\in C_2}P_i^k.
\]
\end{itemize}
Suppose we have defined $P_1^k,P_2^k,P_3^k,\ldots$ with the above properties. We do the job for $k+1$. The second condition implies that
\[
P_1^k[k+1]<P_2^k[k+1]<\cdots<P_m^k[k+1]<\cdots\cdot
\]
Now by Hindman's theorem there are $B_1<B_2<\cdots<B_m<\cdots$ in $\mathcal{P}_{f}(\mathbb{N})$ such that if we put
\[
b_1=\displaystyle\sum_{i\in B_1}P^k_i[k+1],b_2=\displaystyle\sum_{i\in B_2}P^k_i[k+1],\ldots,b_m=\displaystyle\sum_{i\in B_m}P^k_i[k+1],\dots
\]
then $\mathbf{c}$ has a constant value on $FS(\{b_1,b_2,\dots\})$, which we denote it by $\alpha$. Now we set
\[
P_1^{k+1}=\displaystyle\bigoplus_{i\in B_1} P_i^{k}, P_2^{k+1}=\displaystyle\bigoplus_{i\in B_2} P_i^{k},\ldots, P_m^{k+1}=\displaystyle\bigoplus_{i\in B_m} P_i^{k},\ldots.
\]
as well as we set
\[
\alpha_{1}^{k+1}=\alpha_{1}^{k}, \alpha_{2}^{k+1}=\alpha_{2}^{k},\dots, \alpha_{k}^{k+1}=\alpha_{k}^{k},\alpha_{k+1}^{k+1}=\alpha.
\]
We check the conditions (a) and (b) for $k+1$. Let $C\in\mathcal{P}_{f}(\mathbb{N})$ and $1\leq s\leq k+1$, hence we have
\[
\displaystyle\bigoplus_{i\in C}P_i^{k+1}\,[s]=\displaystyle\bigoplus_{i\in C}\displaystyle\bigoplus_{j\in B_i}\!P_j^{k}\,[s]=\displaystyle\bigoplus_{i\in D}P_i^{k}\,[s],
\]
where $D=\bigcup_{i\in C}B_i$. Suppose $1\leq s\leq k$, from the induction hypothesis it follows that
\begin{equation}\label{eq1}
\mathbf{c}\big{(}\displaystyle\bigoplus_{i\in D}P_i^{k}\,[s]\big{)}=\alpha_s^k=\alpha_s^{k+1}.
\end{equation}
Also for $s=k+1$ we have
\[
\displaystyle\bigoplus_{i\in C}\displaystyle\bigoplus_{j\in B_i}\!P_j^{k}\,[k+1]=\displaystyle\sum_{i\in C}\displaystyle\sum_{j\in B_i} P_j^{k}[k+1]=\displaystyle\sum_{i\in C}b_i\in FS(\{b_1, b_2,\dots\}),
\]
which implies that
\begin{equation}\label{eq2}
\mathbf{c}\big{(}\displaystyle\bigoplus_{i\in C}\displaystyle\bigoplus_{j\in B_i}\!P_j^{k}\,[k+1]\big{)}=\mathbf{c}\big{(}\displaystyle\sum_{i\in C}b_i\big{)}=\alpha=\alpha_{k+1}^{k+1}.
\end{equation}
Now putting (\ref{eq1}) and (\ref{eq2}) together we deduce
\[
\mathbf{c}\big{(}\displaystyle\bigoplus_{i\in C}P_i^{k+1}\,[s]\big{)}=\alpha_s^{k+1}
\]
for $1\leq s\leq k+1$. This finishes the proof of the condition (a). Now we turn to checking (b). Let $C_1<C_2$ be in $\mathcal{P}_{f}(\mathbb{N})$. We must show that
\[
\displaystyle\bigoplus_{i\in C_1}P_i^{k+1}\prec\displaystyle\bigoplus_{i\in C_2}P_i^{k+1}
\]
which is equivalent to
\begin{equation}\label{eq3}
\displaystyle\bigoplus_{i\in C_1}\displaystyle\bigoplus_{j\in B_i}\!P_j^{k}\prec\displaystyle\bigoplus_{i\in C_2}\displaystyle\bigoplus_{j\in B_i}\!P_j^{k}.
\end{equation}
Letting $D_1=\bigcup_{i\in C_1}B_i$, $D_2=\bigcup_{i\in C_2}B_i$, we get $D_1<D_2$ and (\ref{eq3}) becomes
\[
\displaystyle\bigoplus_{i\in D_1}P_i^{k}\prec\displaystyle\bigoplus_{i\in D_2}P_i^{k}
\]
which is exactly our induction hypothesis. This proves the condition (b).

Now consider $P_1^n[1],P_1^n[2],\dots,P_1^n[n]$ and recall that $n=W(l,c)$. By construction we have
\[
\mathbf{c}(P_1^n[1])=\alpha^n_1,\dots,\mathbf{c}(P_1^n[n])=\alpha^n_n.
\]
Through induced coloring, it follows from van der Waerden's theorem that there exist $\gamma\in[c]$ and positive integers $a, d$ such that
\[
\alpha^n_a=\alpha^n_{a+d}=\cdots=\alpha^n_{a+(l-1)d}=\gamma.
\]
We define the desire arithmetic progressions $Q_i, i\in\mathbb{N}$ as follows
\[
Q_i=\big{\{}P_i^n[a],P_i^n[a+d],\dots, P_i^n[a+(l-1)d]\big{\}}.
\]
It is easily seen by condition (b) that $Q_1\prec Q_2\prec Q_3\prec \cdots\,\cdot$ Also for all $C\in\mathcal{P}_f(\mathbb{N})$ and all $1\leq s\leq l$ we have
\[
\mathbf{c}\big{(}\displaystyle\bigoplus_{i\in C}\!Q_i\,[s]\big{)}=\mathbf{c}\big{(}\displaystyle\sum_{i\in C}Q_i[s]\big{)}=\mathbf{c}\big{(}\displaystyle\sum_{i\in C}P_i^n[a+(s-1)d]\big{)}=\alpha^n_{a+(s-1)d}=\gamma.
\]
This finishes the proof of Theorem \ref{maintheorem}.
\end{proof}
Now we turn to the two-fold generalization of van der Waerden-Brauer's theorem and Hindman's theorem.
\begin{theorem}\label{secondtheorem}
Let $c$ and $l\geq 3$ be positive integers. Let $\mathbf{c}$ be a $c$-coloring of $\mathbb{N}$, then there are $l$-term arithmetic progressions $Q_1, Q_2,Q_3,\dots$ such that
\begin{itemize}
\item[(i)] $Q_1\prec Q_2\prec Q_3\prec\cdots$,
\item[(ii)] there is $\gamma\in[c]$ such that for all $C\in\mathcal{P}_{f}(\mathbb{N})$ and all $s\in\{1,\dots,l\}$ we have
\[
\mathbf{c}\big{(}\!\displaystyle\bigoplus_{i\in C}\!Q_i\,[s]\big{)}=\mathbf{c}\big{(}\!\add\displaystyle\bigoplus_{i\in C}\!Q_i\big{)}=\gamma.
\]
\end{itemize}
\end{theorem}
\begin{proof}
We start with $n=WB(l,c)$ and a strictly increasing sequence of positive integers $a_1<a_2<\dots<a_m<\cdots$ with $a_{m+1}>n(a_1+\dots+a_m)$. For $i\in\mathbb{N}$, We put
$P^{0}_{i}=\{a_i, a_i+a_i,\dots,a_i+(n-1)a_i\}$. In this case for all $1\leq k\leq n$ and all $C\in\mathcal{P}_f(\mathbb{N})$ we will have
\begin{equation}\label{eq4}
\add\displaystyle\bigoplus_{i\in C}\!P^k_i=\displaystyle\bigoplus_{i\in C}\!P^k_i\,[1].
\end{equation}
We prove (\ref{eq4}) by induction on $k$. First observe that
\begin{eqnarray*}
\add\displaystyle\bigoplus_{i\in C}\!P^0_i=\displaystyle\bigoplus_{i\in C}\!P^0_i\,[2]-\displaystyle\bigoplus_{i\in C}\!P^0_i\,[1]
                                                            &=&\displaystyle\sum_{i\in C}P^0_i[2]-\displaystyle\sum_{i\in C}P^0_i[1]\\
                                                            &=&\displaystyle\sum_{i\in C}(a_i+a_i)-\displaystyle\sum_{i\in C}a_i\\
                                                            &=&\displaystyle\sum_{i\in C}a_i=\displaystyle\sum_{i\in C}P^0_i[1]
                                                            =\displaystyle\bigoplus_{i\in C}\!P^0_i\,[1].
\end{eqnarray*}
Also for $k+1$, recall the subsets $B_i$ in definition of the arithmetic progressions $P^{k+1}_i$, so we have
\begin{eqnarray*}
\add\displaystyle\bigoplus_{i\in C}\!P^{k+1}_i&=&\add\displaystyle\bigoplus_{i\in C}\!\displaystyle\bigoplus_{j\in B_i}\!P^{k}_j
                                                                      =\add\displaystyle\bigoplus_{i\in D }\!P^{k}_i\\
                                              &=&\displaystyle\bigoplus_{i\in D }\!P^{k}_i\,[1]=\displaystyle\bigoplus_{i\in C}\!P^{k+1}_i\,[1]
\end{eqnarray*}
where $D=\bigcup_{i\in C}B_i$. This proves (\ref{eq4}). The proof now proceeds as in the proof of Theorem \ref{maintheorem}, in particular (\ref{new}) can be proved easily for these new $P^0_i$. Now recall $P_1^n[1],P_1^n[2],\dots,P_1^n[n]$ so that for $s\in\{1,\dots,n\}$ and $C\in\mathcal{P}_f(\mathbb{N})$ we have
\[
\mathbf{c}\big{(}\displaystyle\bigoplus_{i\in C}P_i^{n}\,[s]\big{)}=\alpha_s^{n}.
\]
Through induced coloring and this time using $=WB(l,c)$ we obtain $\gamma\in[c]$ and positive integers $a, d$ such that
\[
\alpha^n_d=\alpha^n_a=\alpha^n_{a+d}=\cdots=\alpha^n_{a+(l-1)d}=\gamma.
\]
Again define the desire arithmetic progressions $Q_i, i\in\mathbb{N}$ by
\[
Q_i=\big{\{}P_i^n[a],P_i^n[a+d],\dots, P_i^n[a+(l-1)d]\big{\}}.
\]
Thus for all $C\in\mathcal{P}_f(\mathbb{N})$ we have
\begin{eqnarray*}
\add\displaystyle\bigoplus_{i\in C}\!Q_i&=&\displaystyle\bigoplus_{i\in C}\!Q_i\,[2]-\displaystyle\bigoplus_{i\in C}\!Q_i\,[1]
                                                         =\displaystyle\sum_{i\in C}Q_i[2]-\displaystyle\sum_{i\in C}Q_i[1]\\
                                        &=&\displaystyle\sum_{i\in C}P_i^n[a+d]-\displaystyle\sum_{i\in C}P_i^n[a]=\displaystyle\sum_{i\in C}\big{(}P_i[a+d]-P_i[a]\big{)}\\
                                        &=&\displaystyle\sum_{i\in C}\displaystyle\sum_{t=1}^d\big{(}P_i^n[a+t]-P_i^n[a+(t-1)]\big{)}=\displaystyle\sum_{i\in C}\displaystyle\sum_{t=1}^d\add P_i^n\\
                                        &=&\displaystyle\sum_{i\in C}d.\add P_i^n=d\displaystyle\sum_{i\in C}\add P_i^n=d.\add\displaystyle\bigoplus_{i\in C}\!P^n_i\\
                                        &=&\displaystyle\bigoplus_{i\in C}\!P^n_i\,[1]+(d-1)\add\displaystyle\bigoplus_{i\in C}\!P^n_i=\displaystyle\bigoplus_{i\in C}\!P^n_i\,[d].
\end{eqnarray*}
Note that in the second and third equations from the end we have respectively used (\ref{eq4}) and the easily checked fact $\displaystyle\sum_{i\in C}\add P_i^n=\add\displaystyle\bigoplus_{i\in C}\!P^n_i$. So we conclude that
\[
\mathbf{c}\big{(}\!\add\displaystyle\bigoplus_{i\in C}\!Q_i\big{)}=\mathbf{c}\big{(}\!\displaystyle\bigoplus_{i\in C}\!P^n_i\,[d]\big{)}=\alpha^n_d=\gamma,
\]
and the rest of the proof is the same as the proof of Theorem \ref{maintheorem}.
\end{proof}

\section{Tower Bounds for the Finite Case}

In this section we prove
\begin{theorem}\label{thirdtheorem}
For positive integers $n,c$ and $l\geq3$, let $f(n,l,c)$ be the least positive integer $p$ such that whenever $\mathbf{c}$ is a $c$-coloring of $[p]$, then there are $l$-term arithmetic progressions $Q_1,Q_2,\dots,Q_n$ such that
\begin{itemize}
\item[(i)] $Q_1\prec\cdots\prec Q_n$,
\item[(ii)] $\max (Q_1\oplus\cdots\oplus Q_n)\leq p$,
\item[(iii)] there is $\gamma\in[c]$ such that for all $C\in\mathcal{P}^{+}([n])$ and all $s\in\{1,\dots,l\}$ we have
\[
\mathbf{c}\big{(}\!\displaystyle\bigoplus_{i\in C}\!Q_i\,[s]\big{)}=\gamma.
\]
Then $f(n,l,c)$ is a tower function.
\end{itemize}
\end{theorem}
\begin{proof}
Let $q=W(l,c^{2^{\hind(n,c)}})$, we will show that $f(n,l,c)\leq 2^{q^3}$. So from Gower's elementary bounds for the van der Waerden numbers \cite{gowers} and Theorem \ref{dut}, it follows that $f(n,l,c)$ is a tower function. Suppose that $p\geq 2^{q^3}$ and $\mathbf{c}$ is a $c$-coloring of $[p]$. We show that $p$ satisfies the requirements of the theorem. Put $m=\hind(n,c)$. Let $h_i, 1\leq i\leq m$ be positive integers defined by $h_i=(m+i)+(i-1)q$. For $1\leq i\leq m$, we define the $q$-term arithmetic progressions $P_i$ as follows
\[
P_i=\{2^i,2^i+2^{h_i},2^i+2.2^{h_i},\dots,2^i+(q-1)2^{h_i}\}.
\]
Clearly $P_1\prec P_2\prec\cdots\prec P_m$. We claim that for each $1\leq s\leq q$, the positive integers $P_1[s],P_2[s],\dots,P_m[s]$ are pairwise power-disjoint. Let $1\leq s\leq q$, $2^u\leq q-1< 2^{u+1}$ and $s-1=2^{u_1}+\dots+2^{u_k}$ with $u_1<u_2<\cdots<u_k$, hence $u_k\leq u\leq q-1$. Also from $i\leq m<h_1\leq h_i$ and
\[
P_i[s]=2^{i}+(s-1)2^{h_i}=2^{i}+2^{u_1+h_i}+\dots+2^{u_k+h_i}
\]
it follows that
\[
\pow(P_i[s])\subseteq\{i,h_i,h_i+1,\dots,h_i+(q-1)\}=:A_i
\]
for $1\leq i\leq m$. Now to prove the claim it would be enough to show that $A_1,\dots,A_m$ are pairwise disjoint. In fact we show that
\[
\{1,2,\dots,m\}<A_1-\{1\}<A_2-\{2\}<\cdots<A_m-\{m\}
\]
which easily implies the disjointness of $A_1,\dots,A_m$. First observe that
\[
\min(A_1-\{1\})=h_1=m+1>m.
\]
Also for $1\leq i\leq m-1$ we have
\begin{eqnarray*}
\min(A_{i+1}-\{i+1\})=h_{i+1}&=&(m+i+1)+iq\\
                             &>&(m+i)+(i-1)q+(q-1)\\
                             &=&h_i+(q-1)\\
                             &=&\max(A_{i}-\{i\}),
\end{eqnarray*}
thus the claim is proved. Also we have
\begin{eqnarray*}
\max\displaystyle\bigoplus_{i\in[m]}\!\!P_i=\displaystyle\bigoplus_{i\in[m]}\!\!P_i\,[q]=\displaystyle\sum_{i\in[m]}P_i[q]
                                                                                 &\leq& m2^m+m(q-1)2^{h_m}\\
                                                                                 &\leq& q.2^q+q^{2}.2^{2m+(m-1)q}\\
                                                                                 &\leq& 2^{2q}+q^2. 2^{2q+q^{2}}\\
                                                                                 &\leq& 2^{2q}+2^q. 2^{2q^{2}}\\
                                                                                 &\leq& 2^{q+1}.2^{2q^{2}}\leq2^{q^3}\leq p.
\end{eqnarray*}
Now we define a coloring $\mathbf{c}^{*}$ on $[q]$ as follows. For $u,v\in[q]$, we put $\mathbf{c}^{*}(u)=\mathbf{c}^{*}(v)$ if for all $B\in\mathcal{P}^+([m])$ we have
\[
\mathbf{c}\big{(}\!\bigoplus_{i\in B}\!P_i\,[u]\big{)}=\mathbf{c}\big{(}\!\bigoplus_{i\in B}\!P_i\,[v]\big{)}.
\]
Obviously the number of colors is $c^{2^m-1}$, so from $q=W(l,c^{2^m})$ it follows that there are $a, a+d,\dots, a+(l-1)d$ in $\{1,2,\dots,q\}$ such that
\[
\mathbf{c}^{*}(a)=\mathbf{c}^{*}(a+d)=\cdots=\mathbf{c}^{*}(a+(l-1)d)
\]
which means that for all $B\in\mathcal{P}^+([m])$ and all $k_1,k_2\in\{0,\dots,l-1\}$ we have
\[
\mathbf{c}\big{(}\!\bigoplus_{i\in B}\!P_i\,[a+k_1d]\big{)}=\mathbf{c}\big{(}\!\bigoplus_{i\in B}\!P_i\,[a+k_2d]\big{)}.
\]
We denote the above color by $\pi(B)$. So we have the well-defined function
\[
\pi\colon\mathcal{P}^+([m])\longrightarrow[c].
\]
Now consider the following $m$-elements set of power-disjoint (due to the claim) positive integers
\[
\big{\{}P_1[a],P_2[a],\dots,P_m[a]\big{\}}.
\]
From $m=\hind(n,c)$ we infer that there exist $B_1<B_2<\cdots<B_n$ in $\mathcal{P}^+([m])$ and $\gamma\in[c]$ so that for all $C\in NU\{B_1,\dots,B_n\}$ we have
\[
\pi(C)=\mathbf{c}\big{(}\!\sum_{i\in C}P_i[a]\big{)}=\gamma.
\]
The desired arithmetic progressions $Q_1,\dots,Q_n$ are defined as follows. For $1\leq i\leq n$, we set
\[
Q_i=\big{\{}\!\!\bigoplus_{j\in B_i}\!P_j\,[a],\bigoplus_{j\in B_i}\!P_j\,[a+d],\ldots,\bigoplus_{j\in B_i}\!P_j\,[a+(l-1)d]\big{\}}.
\]
Obviously $Q_1\prec Q_2\prec\cdots\prec Q_n$ and from $B_1<B_2<\cdots<B_n$ it is easily seen that
\[
\max(Q_1\oplus\cdots\oplus Q_n)\leq\max(P_1\oplus\cdots\oplus P_m)\leq p.
\]
Now for $C\in\mathcal{P}^+([n])$ and $1\leq s\leq l$ we have
\begin{eqnarray*}
\mathbf{c}\big{(}\!\bigoplus_{i\in C}\!Q_i\,[s]\big{)}=\mathbf{c}\big{(}\!\sum_{i\in C}Q_i[s]\big{)}
                                                             &=&\mathbf{c}\big{(}\!\sum_{i\in C}\bigoplus_{j\in B_i}\!P_j\,[a+(s-1)d]\big{)}\\
                                                             &=&\mathbf{c}\big{(}\!\sum_{i\in C}\sum_{j\in B_i}P_j[a+(s-1)d]\big{)}\\
                                                             &=&\mathbf{c}\big{(}\!\sum_{i\in D}P_i[a+(s-1)d]\big{)}=\pi(D)=\gamma,
\end{eqnarray*}
where $D=\bigcup_{i\in C}B_i\in NU\{B_1,\dots,B_n\}$. This finishes the proof of the theorem.
\end{proof}


\bibliography{reference}
\bibliographystyle{plain}
\end{document}